\documentclass[11pt,reqno]{amsart}

\usepackage{amsmath,amsfonts,amsthm,amssymb,amsxtra,comment}
\usepackage[colorlinks,citecolor=red,hypertexnames=false]{hyperref} 

\usepackage{color}
\usepackage[shortlabels]{enumitem}
\usepackage{bbm} 
\usepackage{stmaryrd}
\usepackage{nicematrix}
\usepackage{mathrsfs} 
\usepackage{float}
\usepackage{graphicx}
\usepackage{tikz-cd}

\usepackage[margin=1.1in]{geometry}
\usepackage{extarrows}

\usepackage{subcaption}

\usepackage{tikz}
\usepackage{pgfplots}
\pgfplotsset{compat=1.18} 
\usetikzlibrary{calc}

\newtheorem{theorem}{Theorem}[section]

\newtheorem{proposition}[theorem]{Proposition}
\newtheorem{lemma}[theorem]{Lemma}
\newtheorem{corollary}[theorem]{Corollary}

\theoremstyle{definition}

\theoremstyle{remark}

\numberwithin{equation}{section}
\allowdisplaybreaks

\newcommand{\C}{\mathbb{C}}

\renewcommand{\epsilon}{\varepsilon}

\newcommand{\N}{\mathbb{N}}

\newcommand{\cD}{\mathcal{D}}
\newcommand{\cR}{\mathcal{R}}

\newcommand{\bbR}{\mathbb{R}}
\newcommand{\bbC}{\mathbb{C}}

\newcommand{\sE}{\mathsf{E}}

\newcommand{\ess}{\mathrm{ess}}

\renewcommand{\phi}{\varphi}
\newcommand{\R}{\mathbb{R}}

\newcommand{\T}{\mathbb{T}}

\newcommand{\Z}{\mathbb{Z}}

\DeclareFontFamily{U}{mathx}{}
\DeclareFontShape{U}{mathx}{m}{n}{<-> mathx10}{}
\DeclareSymbolFont{mathx}{U}{mathx}{m}{n}
\DeclareMathAccent{\widehat}{0}{mathx}{"70}
\DeclareMathAccent{\widecheck}{0}{mathx}{"71}

  \renewcommand{\mod}{{\rm \, mod\, }}

\newcommand{\vertiii}[1]{{\left\vert\kern-0.25ex\left\vert\kern-0.25ex\left\vert #1
    \right\vert\kern-0.25ex\right\vert\kern-0.25ex\right\vert}}

\newcommand{\comments}[1]{}

\renewcommand{\Re}{\mathrm{Re}}
\renewcommand{\Im}{\mathrm{Im}}

\begin{document}

 \title[Almost periodic solutions of mKdV]{Almost periodic solutions of the defocusing mKdV equation}

\author[L. Li]{Long Li}

\address{L. Li: Texas A\&M University,
College Station, TX 77843, USA}
\email{\href{mailto:longli@tamu.edu}{longli@tamu.edu}}
\thanks{L.\ L.\ was supported by AMS-Simons Travel Grant 2024-2026.}

\author[M. Luki{\'c}]{Milivoje Luki{\'c}}
 
\address{M.\ Luki\'c: Emory University, Atlanta, GA~30322, USA}
\email{\href{milivoje.lukic@emory.edu}{milivoje.lukic@emory.edu}}
\thanks{M.\ L.\ was supported in part by NSF grants DMS--2154563/2626207 and DMS--2453758/2626193.}

\begin{abstract}
We study the Cauchy problem for the defocusing mKdV equation with almost periodic initial data. If the initial data corresponds to a reflectionless Dirac operator which satisfies certain Craig-type conditions on the spectrum, we prove that the Cauchy problem has solution that is almost periodic in spacetime, and that this is the only solution which is locally bounded in a suitable sense.
In particular, the result applies to small analytic quasiperiodic initial data with Diophantine frequencies. 
\end{abstract}

\maketitle

\tableofcontents

\section{Introduction and Main Results}
In this work we will consider the Cauchy problem of the defocusing modified KdV  (mKdV) equation
\begin{align}
\partial_t u & = -u'''+6|u|^2 u', \label{eqn:mKdV} \\
     u(x,0) & =\varphi(x), \label{eqn:cauchyI}
\end{align}
where $\varphi:\R\to\C$ is the initial data and $u'=\partial_x u$. This is an integrable equation: it corresponds to the third equation in the AKNS hierarchy \cite{AKNS,ZS,GH03}, the same hierarchy which contains the  nonlinear Schr\"odinger equation. This hierarchy, in its general form, describes flows on a pair of functions $(p,q)$, and its third equation is 
\[
\left(\begin{aligned}
\partial_{t}p+\frac{1}{4}\partial_x^3p-\frac{3}{2}p\partial_xpq\\
\partial_{t}q+\frac{1}{4}\partial_x^3q-\frac{3}{2}pq\partial_x q
\end{aligned}\right)=0.\]
Restricting that system to $p = \pm \bar q$ gives rise to the defocusing/focusing complex mKdV equations
\[
    \partial_tq=-q'''\pm 6|q|^2q'.
\]
The defocusing mKdV \eqref{eqn:mKdV} corresponds to the case \(p=\overline{q}\). Restricting to real solutions recovers the classical mKdV equation of Miura \cite{Miura1968}:
\[\partial_t q=-q'''\pm 6q^2q'.\]

Well-posedness of the mKdV equation has been extensively studied for decaying initial data
\cite{Tsutsumi1981,Kato1983, KPV1989DMJ, KPV1993CPAM,CKSTT2003JAMS,Guo2009JMPA,Kis2009DIE, GKV2024ForumPi} and for periodic initial data \cite{Bourgain1,Bourgain2,CKSTT2003JAMS,CCT2003AJM,KappelerTopalovmKdV,Molinet,KappelerMolnar,Chapouto}. In particular, Harrop-Griffiths--Killip--Vi\c san \cite{GKV2024ForumPi} showed that mKdV is globally well-posed in $H^s(\mathbb{R})$ for $s >-\frac{1}{2}$. The construction of solutions of mKdV and their long time asymptotics are a subject of continued interest \cite{DZ1993Ann, Griffiths2016CPDE,GGJMM}.

In this paper, we study non-decaying, non-periodic initial data; these solutions cannot be studied in a space such as $H^s(\mathbb{T})$ or $H^s(\mathbb{R})$, but further questions about their structure are motivated by integrability. A canonical class of solutions of mKdV are obtained by an algebro-geometric construction using the Its-Matveev formula \cite{ItsMatveev1975b}, see also \cite{DN1974,Lax1975CPAM,DMN1976,BBEIM1994,GH03,DeConciniJohnson}. Such solutions are known for many integrable equations, and correspond to the case when the Lax operator has spectrum consisting of finitely many intervals and has purely absolutely continuous spectrum. These solutions are spacetime quasiperiodic. Integrability motivates more general questions: whether almost periodic initial data give solutions which are almost periodic in time, or spacetime almost periodic. This was conjectured by P. Lax \cite{Lax1975CPAM} in the setting of the KdV equation, supported by numerical evidence of  M. Hyman that can be found in the appendix of \cite{Lax1975CPAM}.  This problem was popularized by P. Deift \cite{DeiftOpenProblem2008, DeiftOpenProblem2017} and is now known as the \emph{Deift conjecture}; it is particularly well-studied in the KdV setting.
McKean--Trubowitz \cite{MT1976CPAM} studied this problem with periodic initial data whose associated Lax operator has infinitely many gaps in the spectrum and proved almost periodicity in time. For certain limit-periodic initial data with rapid periodic approximants, see \cite{Egorova1992}. Beyond the periodic setting, severe issues arise since the spectrum is typically a Cantor set, with infinitely many gaps. The same issues arose in time-independent questions in the inverse spectral theory of almost periodic operators, and a description of the isospectral torus for suitable infinite-gap sets (regular Widom sets with the DCT property) was developed by Sodin--Yuditskii \cite{SY97} for Jacobi matrices, see also \cite{SY95,EVY2019,BLY2}.  In \cite{BDGL2018DMJ}, the Deift conjecture was confirmed for the KdV equation with a class of almost periodic initial data whose associated Lax operator has absolutely continuous spectrum and the spectrum obeys suitable thickness assumptions. The result was further extended to the full KdV hierarchy \cite{EVY2019,LY2020JFA} and to the defocusing nonlinear Schr\"odinger equation \cite{FLLZ}. It is also known that the conjecture is not true in full generality: Chapouto--Killip--Vi\c{s}an \cite{ChapoutoKillipVisan} found counterexamples to the Deift conjecture in the sense that the solution with almost periodic initial data can become discontinuous in $x$ at some future time.

In this work, we will extend the results of \cite{BDGL2018DMJ,FLLZ} to the defocusing mKdV equation and confirm the Deift conjecture under certain conditions on the spectrum of the Lax operator. In particular, the result applies to the small analytic quasi-periodic initial data with a full measure condition Diophantine frequencies. This analysis is based on the direct and inverse spectral theory of the Lax operator; since the NLS and mKdV are part of the same hierarchy, they have the same Lax operator, and this work can be seen as an extension of our prior joint work with Fillman and Zhou \cite{FLLZ} for the defocusing NLS. We will describe the substantial differences and refer to \cite{FLLZ} where the analysis is analogous. We note also that our work is naturally in the setting of the complex mKdV equation \eqref{eqn:mKdV}, but it obviously includes special cases where the solution is real and solves the real mKdV equation.

The  Lax operator for the defocusing mKdV equation can be taken as the Dirac operator
\begin{align}\label{eqn:DiracOper}
\Lambda_\varphi = \begin{bmatrix}
	i & 0 \\ 0 & -i
\end{bmatrix}	 \frac{d}{dx} + \begin{bmatrix} 0 & \varphi \\ \overline{\varphi} & 0 \end{bmatrix}.
\end{align}
The function $\varphi$ is called the operator data or the potential of $\Lambda_\varphi$. If \(\varphi\in L^\infty(\R)\), the Dirac operator is self-adjoint with operator domain $H^1(\bbR)^2$. Its spectrum \(\mathsf{E}\subset\R\) is an unbounded closed subset. The connected components of \(\R\setminus\mathsf{E}\) consists of at most countably many \emph{spectral gaps} \(G_j=(a_j,b_j), j\in J\subset \N\) such that \(\mathsf{E}=\R\setminus\cup_{j\in J}G_j\). Denote 
\begin{align}
    &\gamma_j=b_j-a_j,\,\eta_{jk}=\mathrm{dist}(G_j,G_k),\label{eqn:gapsizeDist}\\
    &C_j=\sup_{\lambda\in G_j}\left(\prod_{G_k<G_j}\frac{a_k-\lambda}{b_k-\lambda}\prod_{G_j<G_k}\frac{b_k-\lambda}{a_k-\lambda}\right)^{\frac{1}{2}},\label{eqn:Cj}
\end{align}
where we used \(G_k<G_j\) to mean \(G_k\) lies to the left of \(G_j\) in \(\R\). In order to present our main results, we will need the following  \emph{Craig-type conditions}:
\begin{align}
    &\sup_{j\in J}C_j^2(1+\eta_{j0}^3)\sum_{\ell\neq j}\frac{\gamma_j^{1/2}\gamma_\ell^{1/2}}{\eta_{j\ell}}<\infty,\label{eqn:craigCond-1}\\
    &\sum_{\ell}C_{\ell}^2\gamma_{\ell}^{1/2}(1+\eta_{0\ell}^2)<\infty.\label{eqn:craigCond-3}
\end{align}

Our first main result is as follows.
\begin{theorem}\label{thm:main1}
    Assume that \(\varphi\) is uniformly almost periodic and its associated Dirac operator $\Lambda_\varphi$ has purely absolutely continuous spectrum that obeys the Craig-type conditions \eqref{eqn:craigCond-1} and \eqref{eqn:craigCond-3}. Then the Cauchy problem \eqref{eqn:mKdV}, \eqref{eqn:cauchyI} has a unique solution \(u(x,t)\) that is almost periodic in both space and time variables.

    Moreover, there exists a continuous map \(\mathcal{M}:\T^\infty\times\T\to\C\), a set of frequencies \(\eta,\eta^{(1)}\in\R^\infty\), phases \(\vartheta_0,\vartheta_1\in\R,\zeta\in\R^\infty,\zeta'\in\R\) such that 
    \[u(x,t):=\mathcal{M}(\zeta+x\eta+t\eta^{(1)},\zeta'+\vartheta_0x+\vartheta_1t)\]
    solves \eqref{eqn:mKdV}, \eqref{eqn:cauchyI}. For any other classical solution \(\tilde{u}(x,t)\) on a time interval \(t\in [0,T]\) with \(\tilde{u}(x,0)=\varphi(x)\) satisfying the local boundedness condition
    \begin{equation}\label{eqn:mKdVbounded}
    \tilde{u},\partial_x\tilde{u},\partial_x^2\tilde{u}\in L^\infty(\R\times[0,T]),
    \end{equation}
    we have \(\tilde{u}=u.\)
\end{theorem}

As an application, the theorem can be applied to small analytic quasiperiodic initial data with Diophantine frequencies:

\begin{corollary}\label{cor:main2}
    Assume that \(\varphi(x)=Q(x\omega),Q\in C^\omega(\T^d,\C),d\geq 2\) is analytic and \(\omega\) is a Diophantine number satisfying \(\inf_{j\in\Z}|\langle n,\omega\rangle-j|\geq\frac{\kappa}{|n|^\tau}\) for any nonzero \(n\in\Z^d.\) For each \(h>0\), there exists $\epsilon_0=\epsilon_0(d,\kappa,\tau,h)>0$ such that if
    \[\sup_{|y|<h}|Q(x+iy)|<\epsilon_0,\]
    then the conclusions of Theorem~\ref{thm:main1} hold.
\end{corollary}
    
\emph{Outline of the paper:}  In Section~\ref{sectionRicatti}, we recall the Schur functions associated with Dirac operators and derive their time dependence for a solution of \eqref{eqn:mKdV}. In Section \ref{sec:dubrovinVector}, we specialize to the reflectionless setting and study the Dubrovin equations that govern the time evolution of Dirichlet eigenvalues. In Section~\ref{sectionAbel}, we study the generalized Abel map and show that it transforms translation and mKdV flows into linear flows. In Section~\ref{sec.proofMain}, we prove Theorem \ref{thm:main1} and Corollary \ref{cor:main2}.

\section{Ricatti equations for the Schur functions} \label{sectionRicatti}
Let $\varphi \in L^\infty(\bbR)$. The Weyl solutions at $z\in \bbC \setminus \sigma_\ess(\Lambda_\varphi)$ are formal solutions of the eigenvector equation
\[\Lambda_\varphi \Psi_\pm=z\Psi_\pm,\]
where \(\Psi_\pm=\begin{bmatrix}\psi^1_\pm\\\psi^2_\pm\end{bmatrix}\in L^2(\R_\pm,\C^2)\); note that each Weyl solution is only required to be square-integrable on one half-line, and by general principles, it is uniquely determined up to a multiplicative constant. Recall the definition of Schur functions
\begin{align}
    s_+=\frac{\psi_+^1}{\psi_+^2},\qquad 
    s_-=\frac{\psi_-^2}{\psi_-^1}.
\end{align}
In a time-dependent setting, for a family of Dirac operators $\Lambda_{u(\cdot,t)}$ parametrized by $t \in I$ with essential spectrum $\sE = \sigma_\ess(\Lambda_{u(\cdot,t)})$ independent of $t$, we obtain a family of Weyl solutions, so we obtain a family of Schur functions $s_\pm$, which are functions of $x\in\bbR$, $t \in I$, $z \in \bbC \setminus \sE$.

For fixed $x,t$, they are meromorphic functions of $z \in \bbC \setminus \sE$ and obey $s_\pm(x,t,z) = 1 / \overline{ s_\pm(x,t,\bar z)}$, and for $z \in \bbC_+$, they obey $\lvert s_\pm(x,t,z)\rvert < 1$. 

From the eigenvector equation, the following Riccati equations are derived 
\cite{EGL,EFGL22}, which describe the $x$-evolution of the Schur functions:
\begin{align}\label{eqn:ricattiSpace}
    &-i\partial_x s_+ = u -2zs_+ + \overline{u}s_+^2,\\
    &-i\partial_x s_- = - u s_-^2  + 2zs_- - \overline{u}.
\end{align}
We now derive their $t$-evolution under the mKdV flow: 
\begin{lemma}\label{lem:RicattiTime}
    Let \(u\) be a solution of \eqref{eqn:mKdV} on a time interval $I$. 
   If the solution satisfies
    \[
    \sup_{t\in I} \sup_{x\in \bbR} \lvert u(x,t) \rvert + \lvert \partial_x^2 u(x,t) \rvert < \infty
    \]
    for some time interval $I$, then at every $t\in I$ and $x\in \bbR$,
    \[\partial_t s_+=b+2as_+-cs_+^2,\]
    \[\partial_t s_-=c-2as_--bs_-^2,\]
    where
    \begin{align*}
a & =\frac{1}{4}(\overline{u}\partial_x u - u\partial_x\overline{u})-\frac{iz}{2}|u|^2-iz^3, \\
b & = -\frac{i}{4}\partial_x^2 u +\frac{i}{2}|u|^2 u - \frac{z}{2}\partial_x u+ iz^2 u, \\
c & =\frac{i}{4}\partial_x^2\overline{u}-\frac{i}{2}|u|^2\overline{u}-\frac{z}{2}\partial_x\overline{u} - iz^2\overline{u}.
\end{align*}
\end{lemma}

\begin{proof}
Let $G_3,F_2,H_2$ denote the differential expressions in Appendix~\ref{sec:AKNS}.  Let \(\widetilde{V}_3=V_3\) be the differential expression in the homogeneous case, then we have
\begin{align}\label{eqn:V3}
    \widetilde{V}_3=i\begin{bmatrix}-G_3&F_2\\-H_2&G_3\end{bmatrix}
\end{align}
where
\[G_3=\frac{i}{4}(p\partial_xq-q\partial_xp)+\frac{z}{2}pq+z^3,\]
\[F_2=\frac{i}{4}\partial_x^2q-\frac{i}{2}pq^2+\frac{z}{2}\partial_xq-iz^2q,\]
\[H_2=-\frac{i}{4}\partial_x^2p+\frac{i}{2}p^2q+\frac{z}{2}\partial_xp+iz^2p.\]
As described in Appendix~\ref{sec:AKNS}, we specialize to the defocusing case by taking $q = i{u}$ and $p = - i \overline{u}$, so we have
\begin{align}\label{eqn:V-3}
\widetilde{V}_3=\begin{bmatrix}
    \frac{1}{4}(\overline{u}\partial_x u -u\partial_x\overline{u})+\frac{z}{2}|u|^2+z^3&-\frac{i}{4}\partial_x^2 u +\frac{i}{2}|u|^2 u- \frac{z}{2}\partial_x u + i z^2 u\\
    \frac{i}{4}\partial_x^2\overline{u} - \frac{i}{2}|u|^2\overline{u}-\frac{z}{2}\partial_x\overline{u} - iz^2\overline{u}&-\frac{1}{4}(\overline{u}\partial_x u - u\partial_x\overline{u})-\frac{z}{2}|u|^2-z^3
\end{bmatrix}.
\end{align}
Let \(\Phi(x,t,z)\) be the solution of the following initial value problem
\begin{align}\label{eqn:combinedEq}
    \left\{
    \begin{aligned}
    &\partial_x\Phi=U(z)\Phi\\
    &\partial_t\Phi=\widetilde{V}_3\Phi\\
    &\Phi(0,0,z)=\binom{s_+(0,0,z)}1,
    \end{aligned}
    \right.
\end{align}
where  \[U(z)=\begin{bmatrix}-iz& iu \\ -i\overline{u} & iz\end{bmatrix}.\]
Existence of the solution is guaranteed by the zero curvature condition \eqref{eq:zeroCur} (c.f. \cite[Remark 2.1]{LY2020JFA}).
Notice that the first equation in \eqref{eqn:combinedEq} is equivalent to the eigenvector equation 
\(\Lambda_{u(\cdot,0)} \Psi=z\Psi\). For $t=0$, the initial condition ensures that $\Phi(x,0,z)$ is a multiple of $\Psi_+(x,0,z)$.
As in the proof of \cite[Lemma 5.4]{FLLZ} (itself an adaptation of Rybkin's argument in the setting of the KdV equation \cite{Rybkin08}), we conclude that \(\Phi_+(x,t,z)\) satisfying \eqref{eqn:combinedEq} are Weyl solutions for $\Lambda_{u(\cdot,t)}$. By the definition of Schur functions, the $t$-evolution of $\Phi$ implies the $t$-evolution of $s_+$,
\begin{align}
\partial_t s_+ = & \left(-\frac{i}{4}\partial_x^2 u +\frac{i}{2}|u|^2u - \frac{z}{2}\partial_x u +i z^2 u \right)+2 \left(\frac{1}{4}(\overline{u}\partial_xu -u\partial_x\overline{u})+\frac{z}{2}|u|^2+z^3 \right)s_+\\&-\left(\frac{i}{4}\partial_x^2\overline{u}- \frac{i}{2}|u|^2\overline{u}-\frac{z}{2}\partial_x\overline{u}- i z^2\overline{u}\right)s_+^2,\end{align}
Analogously, replacing the initial condition by $\Phi(0,0,z) = \binom{1}{s_-(0,0,z)}$, we obtain the $t$-evolution of $s_-$,
\begin{align}
\partial_t s_-=& \left(\frac{i}{4}\partial_x^2\overline{u}- \frac{i}{2}|u|^2\overline{u}-\frac{z}{2}\partial_x\overline{u}- i z^2\overline{u}\right) -2 \left(\frac{1}{4}(\overline{u}\partial_xu -u\partial_x\overline{u})+\frac{z}{2}|u|^2+z^3 \right) s_-\\&- \left(-\frac{i}{4}\partial_x^2 u +\frac{i}{2}|u|^2u - \frac{z}{2}\partial_x u +i z^2 u \right)s_-^2.
\end{align}
\end{proof}

\section{Reflectionless Solutions and Dubrovin Vector Fields}\label{sec:dubrovinVector}

For a potential $\varphi \in L^\infty(\bbR)$, taking $x=0$ we obtain two Schur functions $s_\pm(z) = s_\pm(0,z)$. By a Borg--Marchenko theorem for Dirac operators, the functions $s_\pm(z) := s_\pm(0,z)$ encode the half-line restrictions of $\Lambda_\varphi$ to $\bbR_\pm$. We also note that, by general principles, $s_\pm$ satisfy the normalization condition \(\lim_{y\to \infty} s_\pm(iy)=0\).

For an unbounded closed subset \(\mathsf{E}=\R\setminus\left(\cup_{j}(a_j,b_j)\right)\) whose essential closure (w.r.t.\ Lebesgue measure) is equal to itself, let \(\mathcal{R}(\mathsf{E})\) be the set of potentials \(\varphi\) whose associated Dirac operator \(\Lambda_\varphi\) satisfies the following conditions:
\begin{enumerate}
    \item  \(1-s_+s_-\neq 0\) in \(\R\setminus\mathsf{E}\)
    \item the associated Schur functions satisfy \(\overline{s_+(\xi+i0)}=s_-(\xi+i0)\) for Lebesgue almost every \(\xi\in\mathsf{E}\)
\end{enumerate}
These two properties say that $\sigma(\Lambda_\varphi) = \mathsf{E}$ and that $\Lambda_\varphi$ is reflectionless. 
Let \(\mathcal{S}(\mathsf{E})\) be the set of pairs of Schur functions \((s_+,s_-)\) satisfying above conditions. We equip \(\mathcal{S}(\mathsf{E})\) with the compact open topology induced by the locally uniform convergence. 

\begin{lemma}[\cite{BLY1}]\label{lem:reflectionlessCond}
    Assume that \(\varphi\) is uniformly almost periodic. Then the Dirac operator \(\Lambda_\varphi\) is reflectionless on the absolutely continuous component of the spectrum. 
\end{lemma}
In particular, if \(\Lambda_\varphi\) has purely absolutely continuous spectrum that coincides with \(\mathsf{E}\), then \(\Lambda_\varphi\) is reflectionless on \(\mathsf{E}\) and thus \(\varphi\in\mathcal{R}(\mathsf{E})\). We also recall that whenever the spectrum $\sE$ is weakly homogeneous, for any $\varphi \in \cR(\sE)$, the operator $\Lambda_\varphi$ has purely a.c.\ spectrum \cite{PoltoratskiRemling}. We equip \(\mathcal{R}(\mathsf{E})\) with the strong resolvent topology.

In this section, we derive the Dubrovin vector fields for reflectionless data and prove that they are Lipschitz under the Craig-type conditions, with a suitable choice of metric. Let $R(z)$ be the resolvent function of the Dirac operator, that is, the upper left entry of the matrix Green's function which is the integral kernel of \((\Lambda_\varphi-z)^{-1}\) for $z\in \bbC \setminus \sigma(\Lambda_\varphi)$. The resolvent function is related to the Schur functions \(s_\pm\) by \begin{align}\label{eqn:resolventShur}
    R(z)=i\frac{(1-s_+)(1-s_-)}{1-s_+s_-}.
\end{align}
It is known that \(R(z)\) is real monotone in each spectral gap \((a_j,b_j)\subset \R\setminus \mathsf{E}\). There exists at most one  zero \(\mu_j\in(a_j,b_j)\) of \(R(z)\). Therefore, \(\mu_j\) is either a zero of \(1-s_+\) or a zero of \(1-s_-\). In the former case, let \(\epsilon_j=+1\) and in the latter case, let \(\epsilon_j=-1\). If \(R(z)\) is strictly positive on \((a_j,b_j)\), define \(\mu_j=a_j\) and identify \(\epsilon_j=+1\) with \(\epsilon_j=-1\). If \(R(z)\) is strictly negative, define \(\mu_j=b_j\) and identify \(\epsilon_j=+1\) with \(\epsilon_j=-1\). 

In this way we have defined the map 
\[
\mathcal{B}:\mathcal{R}(\mathsf{E})\to\mathcal{D}(\mathsf{E})
\]
which maps a reflectionless Dirac operator with spectrum $\sE$ to its Dirichlet data $(\mu_j,\epsilon_j)_{j\in J}$, which is an element of the torus of divisors
\[
\cD(\sE) = \prod_{j\in J}[a_j,b_j]\times\{+1,-1\} /_{\substack{(a_j,-1)\sim (a_j,+1)\\ (b_j,-1)\sim (b_j,+1)}}
\]
In the case that $\varphi$ is a finite gap potential, \(\sharp J\) is finite. Otherwise, \(\sharp J\) is countable. As $\cD(\sE)$ is a product of circles, equip $\mathcal{D}(\mathsf{E})$ with the product topology.

We will need a higher (3rd) order trace formula. Define the scalar fields
\begin{align}\label{eqn:scalarField}
    Q_k(y)=\sum_{j}(a_j^k+b_j^k-2\mu_j^k).
\end{align}

\begin{lemma}
The scalar fields $Q_1,Q_2,Q_3$ are well-defined, continuous functions on $\cD(\sE)$ if
\begin{equation}\label{eqnSummabilityForScalarFields}
\sum_j \gamma_j (1+\eta_{j0})^2 < \infty.
\end{equation}
In particular, they are well-defined and continuous whenever \eqref{eqn:craigCond-3} holds.
\end{lemma}

\begin{proof}
The finite gap length condition \(\sum_j\gamma_j<\infty\) implies that \(Q_1\) is well-defined and continuous by the dominated convergence theorem. Since \(|Q_2|\leq 4\sum_j\gamma_j(1+\eta_{0j})\) and \(|Q_3|\leq 12\sum_j\gamma_j(1+\eta_{0j}^2)\), by the same reasoning, both \(Q_2,Q_3\) are well-defined and continuous.

Note that \eqref{eqn:craigCond-3} implies $\sum_j \gamma_\ell^{1/2} (1+\eta_{0\ell}^2) < \infty$. By squaring term-by-term, this implies
$\sum_\ell \gamma_\ell (1+\eta_{0\ell}^2)^2 < \infty$, so the estimate $\eta_{0\ell}^2 \le (1+\eta_{0\ell}^2)^2$ finishes the proof.
\end{proof}

From now on, let us assume that $u$ is a bounded solution of \eqref{eqn:mKdV}. Then the Dirichlet eigenvalues are functions of $x$ and $t$. 
Introduce the angle variable \(y_j\in[0,\pi)\) by the following change of variables
\begin{align}
    &\mu_j=a_j+(b_j-a_j)\sin^2(y_j)\label{eqn:angleVar-1}\\
    &\epsilon_j=\mathrm{sgn}(\sin(2y_j))\label{eqn:angleVar-2}.
    \end{align}
Let the topology on $\mathcal{D}(\mathsf{E})$ be the induced topology by the norm of $y=(y_j)\in\R^J$
\begin{align}\label{eq.vectorNorm}
    \Vert y\Vert=\sup_j\gamma_j^{1/2}\Vert y_j\Vert_{\mathbb{T}}.
\end{align}
From the product representation (c.f. \cite{FLLZ})
\[R(z)=i\prod_{j\in J}\sqrt{\frac{(z-\mu_j)^2}{(z-a_j)(z-b_j)}},\]
the dependence of \(y_j\) on \(x\) can be derived (c. f. \cite[Lemma 5.6]{FLLZ})
\begin{align}\label{eqn:dubrovin-1}
    \partial_x y_j=(-\Re[u] +\mu_j)W_j, 
\end{align}
where \(W_j\) is the product 
\begin{align}\label{eqn:prodW}
W_j=\prod_{k\neq j}\sqrt{\frac{(a_k-\mu_j)(b_k-\mu_j)}{(\mu_k-\mu_j)^2}}.\end{align}
We will derive the dependence of \(y_j\) on \(t\). As a preliminary,  by \cite[Lemma 5.3]{FLLZ}, 
\begin{align}
    &\Re[u] =-\frac{1}{2}Q_1(y),\label{eqn:scalarField-1}\\
    & \partial_x \Im[u] + (\Im[u])^2=\frac{1}{2}Q_2(y).\label{eqn:scalarField-2}
\end{align}

\begin{lemma}\label{lem:scalarField-3}
The following relation holds for \(Q_3\):
\begin{equation}\label{eqn:scalarField-3}
\frac{1}{2}\partial_x^2\Re[u] + \Im[u] \partial_x\Re[u] =\frac{Q_3}{3}-\frac{Q_1^3}{12}.
\end{equation}
\end{lemma}
\begin{proof}
    The asymptotic expansion of \(R(z)\) at the infinity in the cone \(C_\delta=\{z\in\C:\mathrm{arg}[z]\in(\delta,\pi-\delta)\}\) equals
    \[R(z)=i+\frac{i}{2}Q_1z^{-1}+\frac{i}{4}\left(Q_2+\frac{1}{2}Q_1^2\right)z^{-2}+i\left(\frac{Q_3}{6}+\frac{1}{48}Q_1^3+\frac{1}{8}Q_1Q_2\right)z^{-3}+O(|z|^{-4}).\]
    According to \cite[(4.62), (4.63)]{ClarkGesztesy}, \(R(z)=\frac{2}{M_--M_+}\), where 
    \[M_\pm=\pm i+\sum_{k=1}^Nm_{\pm,k}z^{-k}+o(|z|^{-N}).\]
    Taking \(N=3\) and listing the first few terms of \(m_{\pm,k}\),
    \begin{align}
        m_{+,1}&= - \Im[u] + i\Re[u]= iu,\\
        m_{-,1}&= - \Im[u] -i \Re[u] = - i \overline{u},\\
        m_{+,2}&=\frac{i}{2}\left(\partial_xm_{+,1}+m_{+,1}^2-2i m_{+,1}u\right)= - \frac{1}{2}\partial_x u + \frac{i}{2} u^2,\\
        m_{-,2}&=-\frac{i}{2}\left(\partial_xm_{-,1}+m^2_{-,1} + 2im_{-,1}\overline{u}\right)=-\frac{1}{2}\partial_x\overline{u} -\frac{i}{2}\overline{u}^2,\\
        m_{+,3}&=\frac{i}{2}\left(\partial_x m_{+,2}+2m_{+,1}m_{+,2}-(2im_{+,2}+m_{+,1}^2)\Re[u] + 2m_{+,2}\Im[u]\right)\\
        &=-\frac{i}{4}\partial_x^2 u - \frac{1}{4}\partial_x u^2 +\frac{i}{2}u^2 \Re[u],\\
        m_{-,3}&=-\frac{i}{2}\left(\partial_x m_{-,2}+2m_{-,1}m_{-,2}-(-2im_{-,2}+m_{-,1}^2) \Re[u] + 2m_{-,2}\Im[u]\right)\\
        &= \frac{i}{4}\partial_x^2\overline{u}-\frac{1}{4}\partial_x\overline{u}^2 - \frac{i}{2}\overline{u}^2 \Re[u].
    \end{align}
    We need to compare the coefficients of the term \(z^{-3}\) in the asymptotic expansion of \(R(z)=\frac{2}{M_--M_+}\). Since \(M_--M_+=-2i+\sum_{k=1}^3(m_{-,k}-m_{+,k})z^{-k}+o(|z|^{-4})\),
    \begin{align}
    \frac{2}{M_--M_+}-i&=\frac{1}{2}\left(\sum_{k=1}^3(m_{-,k}-m_{+,k})z^{-k}\right)\Bigg\{1+\left(-\frac{i}{2}\sum_{k=1}^2(m_{-,k}-m_{+,k})z^{-k}\right)\\
    &+\left(-\frac{i}{2}(m_{-,1}-m_{+,1})z^{-1}\right)^2
     \Bigg\}+o(|z|^{-3})
    \end{align}
Therefore, the coefficient of \(z^{-3}\) equals
\begin{align}&2i\left(\frac{Q_3}{6}+\frac{Q_1^3}{48}+\frac{Q_1Q_2}{8}\right)\\&\quad=(m_{-,3}-m_{+,3})-i(m_{-,1}-m_{+,1})(m_{-,2}-m_{+,2})-\frac{1}{4}(m_{-,1}-m_{+,1})^3\\
&\quad=i\Bigg(\frac{1}{2}\partial_x^2\Re[u] + \Im[u]\partial_x\Re[u] -  \Re[u]\partial_x\Im[u] - \Re[u] \Im[u]^2  - \Re[u]^3\Bigg)\\
&\quad=i\Bigg(\frac{1}{2}\partial_x^2\Re[u] + \Im[u]\partial_x \Re[u] +\frac{Q_1Q_2}{4}+\frac{1}{8}Q_1^3\Bigg).\end{align}
It follows that
\[
Q_3=\frac{1}{4}Q_1^3+3\left(\frac{1}{2}\partial_x^2\Re[u] + \Im[u]\partial_x\Re[u]\right). \qedhere
\]
\end{proof}

\begin{lemma}\label{lem:dubrovin-2} 
Assume that $u$ is a solution of the mKdV on a time interval $[0,T]$ with initial data in $\cR(\sE)$, satisfying     \begin{equation}\label{eqn:mKdVboundedq}
    {u},\partial_x{u},\partial_x^2 {u}\in L^\infty(\R\times[0,T]).
    \end{equation}
    Then $\Lambda_{u(t)} \in \cR(\sE)$ for all $t \in [0,T]$. 
Suppose also that \(\mu_j \notin  \{a_j,b_j\}\) and that \eqref{eqnSummabilityForScalarFields} holds. Then
    \begin{equation}
\partial_ty_j=\Bigg(
    \frac{Q_3}{6}+\frac{Q_1^3}{48}+\frac{Q_1Q_2}{8}+\frac{\mu_jQ_2}{4}+\frac{\mu_j Q_1^2}{8}+\frac{\mu_j^2Q_1}{2}+\mu_j^3
    \Bigg)W_j.
    \end{equation}
\end{lemma}
\begin{proof}
The preservation of the reflectionless property follows as in \cite{FLLZ} by an adaptation of an argument of Rybkin \cite{Rybkin08} in the Schr\"odinger case; see also \cite{LY2020JFA}. 
    By the definition, for \(y_j\notin\{0,\frac{\pi}{2}\}\), we have \(s_{\epsilon_j}(x,t,\mu_j)=1.\) Since \[R(z)=i(1-s_+)(1-s_-)/(1-s_+s_-),\] 
    applying Lemma \ref{lem:RicattiTime} gives 
    \[\partial_t R|_{z=\mu_j}=-i\partial_t s_{\epsilon_j}=-i\epsilon_j(b-c+2a)\]
    where \(a,b,c\) are the functions given in Lemma \ref{lem:RicattiTime}. Using the formulas for \(a,b,c\), we have
    \begin{align*}\partial_t R|_{z=\mu_j}=&\epsilon_j\Bigg(-\frac{1}{2}\partial_x^2\Re[u]+|u|^2\Re[u] - \mu_j\partial_x\Im[u] + 2\mu_j^2\Re[u] \\ & \quad +(\Re[u]\partial_x\Im[u] - \Im[u]\partial_x \Re[u])-\mu_j|u|^2-2\mu_j^3\Bigg).\end{align*}
    In order to represent this by the scalar fields \(Q_1,Q_2,Q_3\), we first rewrite it as follows,
    \begin{align*}\partial_tR|_{z=\mu_j}=&\epsilon_j\Bigg(-\frac{1}{2}\partial_x^2\Re[u] - \Im[u]\partial_x\Re[u] +\Re[u]^3+\Re[u](\partial_x\Im[u]+\Im[u]^2)\\& \quad -\mu_j( \partial_x\Im[u]+\Im[u]^2)-\mu_j\Re[u]^2+2\mu_j^2\Re[u]-2\mu_j^3\Bigg).\end{align*}
    Plugging in \eqref{eqn:scalarField-1}, \eqref{eqn:scalarField-2} and \eqref{eqn:scalarField-3}, we have 
    \[\partial_tR\vert_{z=\mu_j}=\epsilon_j\Bigg(
    -\frac{Q_3}{3}-\frac{Q_1^3}{24}-\frac{Q_1Q_2}{4}-\frac{\mu_jQ_2}{2}-\frac{\mu_j Q_1^2}{4}-\mu_j^2Q_1-2\mu_j^3
    \Bigg).\]
    For \(\mu_j\in (a_j,b_j)\), by the chain rule,
    \[\partial_t\mu_j=i\partial_t R(z)|_{z=\mu_j}\sqrt{(a_j-\mu_j)(b_j-\mu_j)} W_j 
    \]
    (recall that $W_j$ is defined by \eqref{eqn:prodW}).    Since 
    \[\partial_t\mu_j=(b_j-a_j)\sin(2y_j)\partial_ty_j=2\epsilon_j\sqrt{(b_j-\mu_j)(\mu_j-a_j)}\partial_t y_j,
    \]
    it follows that 
    \begin{align}\partial_t y_j&=\epsilon_j\left(-\frac{1}{2}\partial_t R|_{z=\mu_j}\right) W_j \\
    &=\Bigg(
    \frac{Q_3}{6}+\frac{Q_1^3}{48}+\frac{Q_1Q_2}{8}+\frac{\mu_jQ_2}{4}+\frac{\mu_j Q_1^2}{8}+\frac{\mu_j^2Q_1}{2}+\mu_j^3
    \Bigg) W_j. 
    \end{align}
For the case \(\mu_j=E\in\{a_j,b_j\}\), the arguments of \cite[Lemma 5.1]{FLLZ} apply to show that this happens in a discrete set. Then using \cite[Lemma 5.2]{FLLZ} and \cite[Lemma 5.7]{FLLZ}, the  Dubrovin vector fields extend to the boundaries. The condition \(\sum_{k\neq j,\eta_{jk}<1}\frac{\gamma_k}{\eta_{jk}}<\infty\) required by \cite{FLLZ} is clearly satisfied by \eqref{eqn:craigCond-1}.
\end{proof}
Equip the tangent space of \(\mathcal{D}(\mathsf{E})\) with the metric induced by the norm
\(\Vert y\Vert=\sup_{j}\gamma_j^{1/2}|y_j|\).
\begin{proposition}\label{prop:LipVecField}
Suppose that the spectrum \(\mathsf{E}\) satisfies the Craig-type conditions \eqref{eqn:craigCond-1}, \eqref{eqn:craigCond-3}. Then there exist Lipschitz vector fields
    \(\Psi(y)\) and \(\Xi(y)\) on $\cD(\sE)$ such that  for any solution $ u$ of the mKdV on a time interval $[0,T]$ with initial data in $\cR(\sE)$, satisfying     \eqref{eqn:mKdVboundedq}, the Dirichlet data evolve according to   
    \[\partial_xy=\Psi(y),\qquad \partial_ty=\Xi(y).\]
\end{proposition}

\begin{proof}
Let \(\Psi\) be defined as
\(\partial_xy_j=\Psi_j(y)\) with \(\Psi_j(y)\) be given as \eqref{eqn:dubrovin-1}. Let \(\partial_ty_j=\Xi_j(y)\) with \(\Xi_j\) be given in Lemma \ref{lem:dubrovin-2}.
The rest of the proof follows the lines in \cite[Proposition 5.8]{FLLZ}. We include the argument for completeness.

The proof for \(\Psi\) remains to be the same as the case in \cite[Proposition 5.8]{FLLZ} using condition \eqref{eqn:craigCond-1}. We only need to prove the Lipschitz property for \(\Xi\). By the definition of \(\Xi_j\) and Lemma \ref{lem:dubrovin-2},
\[\partial_t y_j=\Xi_j(y)=\partial_ty_j=\Bigg(
    \frac{Q_3}{6}+\frac{Q_1^3}{48}+\frac{Q_1Q_2}{8}+\frac{\mu_jQ_2}{4}+\frac{\mu_j Q_1^2}{8}+\frac{\mu_j^2Q_1}{2}+\mu_j^3
    \Bigg)W_j,\]
    where \(W_j\) is defined in \eqref{eqn:prodW}. Note that
    \begin{align*}
    \frac{\partial W_j}{\partial y_j}=&\frac{\gamma_j\sin(2y_j)}{2}\sum_{k\neq j}\left(\prod_{k'\neq j,k}\sqrt{\frac{(a_{k'}-\mu_j)(b_{k'}-\mu_j)}{(\mu_{k'}-\mu_j)^2}}\right)\\
    &\times\frac{(b_k-\mu_j)(\mu_k-a_k)+(a_{k}-\mu_j)(\mu_k-b_k)}{\mathrm{sgn}\big[\mu_j-\mu_k\big](\mu_k-\mu_j)^2\sqrt{(a_k-\mu_j)(b_k-\mu_j)}}
    \end{align*}
    and
    \[\left|\frac{\partial W_j}{\partial y_j}\right|\leq C_j^2\gamma_j\sum_{\ell}\frac{\gamma_\ell}{\eta_{j\ell}^2}.\]
    For \(k\neq j\), one computes that
    \[\frac{\partial W_j}{\partial y_k}=\prod_{\ell\neq j}\sqrt{\frac{(a_\ell-\mu_j)(b_\ell-\mu_j)}{(\mu_\ell-\mu_j)^2}}\frac{\gamma_k\sin(2y_k)}{\mu_j-\mu_k}=W_j\frac{\gamma_k\sin(2y_k)}{\mu_j-\mu_k}\]
    and therefore
    \[\left|\frac{\partial W_j}{\partial y_k}\right|\leq\frac{C_j^2\gamma_k}{\eta_{jk}}.\]
    Note that by the relation \(\mu_j=a_j+\gamma_j\sin^2(y_j)\), \(\frac{d\mu_j}{dy_j}=\gamma_j\sin(2y_j)\).
    For the case \(k=j\),
    \[\frac{\partial \Xi_j}{\partial y_j}=\Bigg(
    \frac{Q_3}{6}+\frac{Q_1^3}{48}+\frac{Q_1Q_2}{8}+\frac{\mu_jQ_2}{4}+\frac{\mu_j Q_1^2}{8}+\frac{\mu_j^2Q_1}{2}+\mu_j^3
    \Bigg)\partial_{y_j}W_j,\]
    since \[\partial_{y_j}\Bigg(
    \frac{Q_3}{6}+\frac{Q_1^3}{48}+\frac{Q_1Q_2}{8}+\frac{\mu_jQ_2}{4}+\frac{\mu_j Q_1^2}{8}+\frac{\mu_j^2Q_1}{2}+\mu_j^3
    \Bigg)=0.\]
    For the case \(k\neq j\), one computes that
    \begin{align}\frac{\partial \Xi_j}{\partial y_k}=&\left(-\mu_k^2-\mu_j^2-\frac{Q_1^2}{8}-\frac{Q_2}{4}-\frac{1}{2}Q_1(\mu_j+\mu_k)\right)W_j\gamma_k\sin(2y_k)\\
    &+\Bigg(
    \frac{Q_3}{6}+\frac{Q_1^3}{48}+\frac{Q_1Q_2}{8}+\frac{\mu_jQ_2}{4}+\frac{\mu_j Q_1^2}{8}+\frac{\mu_j^2Q_1}{2}+\mu_j^3
    \Bigg)\partial_{y_k}W_j.\end{align}
    Clearly, \(|W_j|\leq C_j^2.\)
    It follows that 
    \begin{equation}\label{eqn:jComponentEst}
    \begin{aligned}
    \left|\frac{\partial \Xi_j}{\partial y_j}\right|\leq &\Bigg(\frac{\Vert Q_3\Vert_\infty}{6}+\frac{\Vert Q_1\Vert_\infty}{48}(\Vert Q_1\Vert_\infty^2+6\Vert Q_2\Vert_\infty+6|\mu_j|\Vert Q_1\Vert_\infty+24|\mu_j|^2)\\&+|\mu_j|(\mu_j^2+\frac{\Vert Q_2\Vert_\infty}{4})\Bigg)C_j^2\gamma_j\sum_\ell\frac{\gamma_\ell}{\eta_{j\ell}^2}
    \end{aligned}
    \end{equation}
    and 
    \begin{equation}\label{eqn:kComponentEst}
    \begin{aligned}
    \left|\frac{\partial \Xi_j}{\partial y_k}\right|&\leq\left(\mu_k^2+\mu_j^2+\frac{1}{8}\Vert Q_1\Vert_\infty(\Vert Q_1\Vert_\infty+4|\mu_j|+4|\mu_k|)\right)\gamma_k C_j^2\\
    &+\Bigg(\frac{\Vert Q_3\Vert_\infty}{6}+\frac{\Vert Q_1\Vert_\infty}{48}(\Vert Q_1\Vert_\infty^2+6\Vert Q_2\Vert_\infty+6|\mu_j|\Vert Q_1\Vert_\infty+24|\mu_j|^2)+|\mu_j|(\mu_j^2+\frac{\Vert Q_2\Vert_\infty}{4})\Bigg)\\
    &\times \frac{C_j^2\gamma_k}{\eta_{jk}}.
    \end{aligned}
    \end{equation}
    Let \(y,\tilde{y}\) be given vectors, 
    \begin{align}
        \Vert\Xi(y)-\Xi(\tilde{y})\Vert&=\sup_j\gamma_j^{1/2}\Vert\Xi_j(y)-\Xi_j(\tilde{y})\Vert\leq \sup_j\gamma_j^{1/2}\sum_k\left\Vert\frac{\partial \Xi_j}{\partial y_k}\right\Vert_\infty\Vert y_k-\tilde{y}_k\Vert\\
        &\leq \Vert y-\tilde{y}\Vert\sup_{j}\sum_k\gamma_j^{1/2}\gamma_k^{-1/2}\left\Vert\frac{\partial\Xi_j}{\partial y_k}\right\Vert_\infty.\label{eqn:normLip}
    \end{align}
    We will estimate the right hand side of \eqref{eqn:normLip} for \(k=j\) and the sum of all \(k\neq j.\) For \(k=j,\) by \eqref{eqn:jComponentEst}, we only need to show that \(\sup_jC_j^2|\mu_j|^3\sum_\ell\frac{\gamma_j\gamma_\ell}{\eta_{j\ell}^2}<\infty.\) This is guaranteed by the Craig-type condition \eqref{eqn:craigCond-1}. For the sum
    \(\sum_{k\neq j}\gamma_j^{1/2}\gamma_k^{-1/2}\left\Vert\frac{\partial\Xi_j}{\partial y_k}\right\Vert_\infty\), by \eqref{eqn:kComponentEst},
    we need to control the following sums 
    \[\sup_jC_j^2(1+\eta_{0j}^2)\sum_{k\neq j}\gamma_{j}^{1/2}\gamma_k^{1/2}(1+\eta_{0k}^2),\,\sup_j C_j^2(1+\eta_{0j}^3)\sum_{k\neq j}\frac{\gamma_j^{1/2}\gamma_k^{1/2}}{\eta_{jk}}.\] 
    The second sum is clearly finite by condition \eqref{eqn:craigCond-1}. The first sum is bounded by condition \eqref{eqn:craigCond-1} and \eqref{eqn:craigCond-3}. In fact, we can see from condition \eqref{eqn:craigCond-3} that \(\gamma_j\) is a decaying sequence since \(C_j>1.\)
    Moreover, \(\gamma_j\) decays faster than \(\eta_{0j}^4\) grows. Therefore, since \(|\mu_k|\leq a_0+\gamma_0+\eta_{0k}+\gamma_k\) and \(\gamma_k\) decays, we only need to control the first sum mentioned above. Due to condition \eqref{eqn:craigCond-1}, we only need to control 
    \(\sup_jC_j^2\sum_{k\neq j}\gamma_j^{1/2}\gamma_k^{1/2}\eta_{0k}^2.\) This is finite by \eqref{eqn:craigCond-3}.
\end{proof}

\section{The Generalized Abel Map}\label{sectionAbel}

In this section, we will consider the character group of the fundamental domain and the linear behavior of the character along a solution of mKdV. The central object, often called the generalized Abel map, was introduced in the Jacobi setting in \cite{SY97} and in continuum one-dimensional settings in \cite{EVY2019, BLY2}.

Let \(\Omega=\C\setminus\mathsf{E}\). Let \(\pi_1(\Omega)\) be the fundamental group and \(\pi_1(\Omega)^*\) be the associated character group. That is, for any \(\sigma_1,\sigma_2\in\pi_1(\Omega),\alpha\in\pi_1(\Omega)^*,\) we have
\[\alpha:\pi_1(\Omega)\to \T, \,\alpha(\sigma_1\sigma_2)=\alpha(\sigma_1)\alpha(\sigma_2).\]
The topology of \(\pi_1(\Omega)^*\) is induced by the norm \(\Vert\alpha\Vert=\sum_{j\in J}\frac{1}{2^j}|\alpha_j|.\)

Let $\omega(\cdot, F)$ be the harmonic measure of the domain \(\Omega\) for each subset \(F\subset\mathsf{E}=\partial\Omega\). Fix a reference point \(\xi_*\) in \((a_0,b_0)\), let \(\sigma_k\)
be a loop issuing from \(\xi_*\) and enters the lower half plane through \((a_k,b_k)\) and then returns to \(\xi_*\). Let \(\mathsf{E}_k\) be the subset of \(\mathsf{E}\) enclosed by the loop \(\sigma_k\). Denote \(\omega_k=\omega(\cdot,\mathsf{E}_k)\) and \(\omega_\infty=\omega (\cdot,\mathsf{E}_\infty)\), where \(\mathsf{E}_\infty\) is the unbounded subset of \(\mathsf{E}\) that lies to the right of \(\xi_*\). Then the generalized Abel map in the literature of AKNS hierarchy is defined as 
\begin{align}
    &\mathcal{A}:\mathcal{D}(\mathsf{E})\to\pi_1(\Omega)^*\times\T,\\
    &D\mapsto\left(\cdots,\sum_{j}\frac{\epsilon_j}{2}\int_{a_j}^{\mu_j}\omega_k\mod\Z,\cdots;\sum_{j\in J}-\frac{\epsilon_j}{2}\int_{a_j}^{\mu_j}\omega_\infty\mod\Z\right)
\end{align}
It is useful to distinguish the two components of the map $\mathcal{A}=(\mathcal{A}_c,\mathcal{A}_r)$, where the first \emph{character component} $\mathcal{A}_c:\mathcal{D}(\mathsf{E})\to\pi_1(\Omega)^*$ and the second \emph{rotation component} $\mathcal{A}_r:\mathcal{D}(\mathsf{E})\to\T$ are defined as
\[\mathcal{A}_c(D)=\left(\cdots,\sum_{j}\frac{\epsilon_j}{2}\int_{a_j}^{\mu_j}\omega_k\mod\Z,\cdots\right)\in\pi_1(\Omega)^*,\]
\[\mathcal{A}_r(D)=\left(\sum_{j\in J}-\frac{\epsilon_j}{2}\int_{a_j}^{\mu_j}\omega_\infty\mod\Z\right)\in\T.\]
Note that the map $\mathcal{A}_c$ is well-defined in a Widom domain \cite{SY97}, but the map $\mathcal{A}_r$ is well defined in a Widom domain satisfying a finite gap length condition \cite{FLLZ}.
\begin{lemma}[\cite{BLY2}]\label{lem:abelMapIsHomeo}
The generalized Abel map \(\mathcal{A}\) is a homeomorphism if \(\mathsf{E}\) is homogeneous, that is,
\[\inf_{x\in\mathsf{E}}\inf_{h\in (0,\mathrm{diam}(\mathsf{E}))}\frac{|\mathsf{E}\cap(x-h,x+h)|}{2h}>0.\]
\end{lemma}
Notice that our Craig-type conditions \eqref{eqn:craigCond-1} and \eqref{eqn:craigCond-3} implies that 
\[\sup_{j\in J}\left(\sum_{\ell\in J,\ell\neq j}\frac{\gamma_j^{1/2}\gamma_\ell^{1/2}}{\eta_{j\ell}}\right)+\sup_j\gamma_j<\infty.\]
According to \cite[Lemma A.1]{BDGL2018DMJ}, for the unbounded set $\mathsf{E}$, the conditions \eqref{eqn:craigCond-1} and \eqref{eqn:craigCond-3} implies that $\mathsf{E}$ is homogeneous. We will need this observation in Section \ref{sec.proofMain}. 

We will also need the remarkable linearization property of the generalized Abel map.
\begin{theorem}\label{thm:linearAbel} Suppose the Craig-type conditions \eqref{eqn:craigCond-1} and \eqref{eqn:craigCond-3} hold. Then
the Abel map $\mathcal{A}$ maps the  mKdV and translation flow $u(x_0,t_0)\mapsto u(x_0+x,t_0+t)$ into a linear flow.
\end{theorem}

\begin{proof}
    We will first work on the character coordinate $\mathcal{A}_c$ in $\pi_1(\Omega)^*$ and then the rotation coordinate $\mathcal{A}_r$ in $\T.$ In particular, their linearization requires different conditions.
    
    For $k=0,1,2,\cdots,$ let $\Theta_k(\lambda)$ be the Abelian integrals of the second kind of order $k$ such that 
    \[\mathrm{Im}\Theta_k(\lambda)=\mathrm{Im}\lambda^{k+1}+\int_{\R\setminus\mathsf{E}}G(\xi,\lambda)d\xi^{k+1}.\]
    Then we need to show that the character component of the Abel map $\mathcal{A}_c:\mathcal{D}(\mathsf{E})\to\pi_1(\Omega)^*$ conjugates the mKdV and translation flow into a linear flow. Moreover, the $B$-periods of $\Theta_k,k=0,2$ give the full set of frequencies of this flow. The case $k=0,1$ was proved by \cite[Theorem 6.2]{FLLZ}. The rotation component for $k=0,1$ $\mathcal{A}_r$ was proved by \cite[Theorem 6.4]{FLLZ}. We only need to prove that the images  of $\mathcal{A}_c,\mathcal{A}_r$ are linear  for the case $k=2.$

    \emph{The character component:} If $\R\setminus\mathsf{E}$ consists of finitely many gaps, the linearization property was proved for the entire hierarchy; compare \cite[(3.215), (3.216)]{GH03} (including $k=0,1,2,\cdots$). For the case of infinitely many gaps, notice that the Craig type conditions \eqref{eqn:craigCond-1},\eqref{eqn:craigCond-3} and the Widom condition already implies that $\mathrm{Im}\Theta_k(\lambda)$ is convergent for any fixed $\lambda$ and for $k=0,1,2$. Indeed, the integral \[\left|\int_{\R\setminus\mathsf{E}}G(\xi,\lambda)d\xi^{k+1}\right|\leq\sum_{j}G(c_j,\lambda)|b_j^{k+1}-a_j^{k+1}|,\]
    where $G(c_j,\lambda)$ is the critical value in $(a_j,b_j)$. Conditions \eqref{eqn:craigCond-1}, \eqref{eqn:craigCond-3} imply that $\sum_j\gamma_j^{1/2}<\infty,\, \sum_{j}\gamma_j\eta_{0j}^4<\infty$. The Widom condition implies that $\sum_jG(c_j,\lambda)<\infty$ for any fixed $\lambda\notin\mathsf{E}$.

    Let $\mathsf{E}_N=\R\setminus\left(\cup_{j\in J_N}(a_j,b_j)\right)$ be a sequence of approximants to $\mathsf{E}$, where $\cdots\subset J_N\subset J_{N+1}\subset\cdots\subset J$ is an exhaustion of the index set $J$ for which we have $\mathsf{E}=\R\setminus\left(\cup_{j\in J}(a_j,b_j)\right)$. Since $\{\mathsf{E}_N\}$ is increasing, $\{\Omega_N=\C\setminus\mathsf{E}_N\}$ is a decreasing sequence of Denjoy domains that are Dirichlet regular, homogeneous and have no internal boundary point. By \cite[Theorem 5.15]{Landof1972}, we have the stability of Dirichlet problem. Since the uniform convergence on compacts of the associated Green's functions $G_N(z,\lambda)$ of $\Omega_N$ to that of $G(z,\lambda)$ of $\Omega=\C\setminus\mathsf{E},\lambda\in\Omega$ implies the convergence of $B$-periods of $\Theta_k^{(N)}$ to those of $\Theta_k$, locally uniformly, we have the convergence of $\mathcal{A}_c^{(N)}$ to $\mathcal{A}_c$. Since $\mathcal{A}_c$ is a majorized by an absolutely uniformly convergent series, from the fact that $\mathcal{A}^{(N)}_c$ linearize the translation and mKdV flow \cite[Theorem 3.33]{GH03}, we conclude that $\mathcal{A}_c$ also linearizes the translation and mKdV flow. 
    
    \emph{The rotation component:} Let $M(\lambda)$ be the symmetric Martin function at the infinity for the domain $\bbC \setminus \sE$. Let us first point out that the Craig type conditions \eqref{eqn:craigCond-1}, \eqref{eqn:craigCond-3}  also imply 
    \[\int_{\R\setminus\mathsf{E}}M(\xi)d\xi^{k+1}<\infty, k=0,2.\]
    Indeed, since $M(\xi)$ grows at most linearly for $\xi$ in each gap of $\mathsf{E}$ \cite[Lemma 3.2]{EGL}, it suffices to show
    \[\sum_{j\in J}(1+\eta_{0j})\gamma_j<\infty,\,\sum_{j\in J}(1+\eta_{0j})|b_j^3-a_j^3|<\infty.\]
    The first inequality clearly follows from \eqref{eqn:craigCond-3}. For the second inequality, we have
    \[(1+\eta_{0j})|b_j^3-a_j^3|\leq 3\gamma_j(1+\eta_{0j})(\eta_{0j}+\gamma_j)^2.\]
    Observe that \eqref{eqn:craigCond-1} and \eqref{eqn:craigCond-3} implies that $\sum_{j\in J}\gamma_j\eta_{0j}^4<\infty$. In order to show the linearization of the rotation component $\mathcal{A}_r$, we follow the lines of \cite{FLLZ}. Note that for the finite approximants $\mathsf{E}_N$, the linearity follows from the exact formula \cite[(3.217), (3.218)]{GH03} in the exponential factors of the expressions of the solution. The identification of the rotation component $\mathcal{A}_r$ with the exponential  factors of the expression of the solution was proved in \cite[Theorem 6.4]{FLLZ}. Indeed, they all arise from the Abelian integrals of the second kind along an open loop. According to \cite[Theorem 4.21]{FLLZ}, the series
    \[\mathcal{A}_r(D)=\sum_{j\in J}-\frac{\epsilon_j}{2}\int_{a_j}^{\mu_j}\omega_\infty\mod\Z\]
    is absolutely uniformly convergent. Let $\mathcal{A}_r^{(N)}$ be the rotation component of $\mathsf{E}_N$, then by dominated convergence theorem, we have point wise $\mathcal{A}_r^{(N)}\to\mathcal{A}_r$. In order to prove linearity of $\mathcal{A}_r$, it suffices to show that the Martin kernels $M^{(N)}(z,\lambda)$ of $\Omega_N$ converges uniformly on compacts away from $\lambda$ of $\Omega$ to $M(z,\lambda)$. Since $M(z,\lambda)=\frac{G(z,\lambda)}{G(i,\lambda)}$, where $i\in\Omega$ is a normalization point, we can apply \cite[Theorem 5.15]{Landof1972} again. Since the image of $\mathcal{A}_r$ is a limit of a sequence of linear flows in the topology of uniform convergence on compacts, it is also a linear flow with the limiting frequencies.
\end{proof}

\section{Proof of Main Theorems}\label{sec.proofMain}
In this section, we prove the main results. We  first establish the existence and uniqueness and then prove the almost periodicity.
\begin{proof}[Proof of Theorem \ref{thm:main1}]
    Let \(\varphi\) be an uniformly almost periodic function such that the Dirac operator \(\Lambda_\varphi \) has purely absolutely continuous spectrum that coincides with \(\mathsf{E}\) which satisfies the Craig-type conditions \eqref{eqn:craigCond-1} and \eqref{eqn:craigCond-3}. By Lemma \ref{lem:reflectionlessCond}, \(\varphi\in\mathcal{R}(\mathsf{E})\). Let \(f=\mathcal{B}(\varphi)\in\mathcal{D}(\mathsf{E})\). Consider the following initial value problem 
    \begin{align}\label{eqn:ivp}\left\{\begin{aligned}&\partial_xy=\Psi(y),\,\partial_t=\Xi(y)\\
    &y(0,0)=f.
    \end{aligned}\right.\end{align}
    By Proposition \ref{prop:LipVecField}, there exists a unique global solution \(y(x,t)\) of \eqref{eqn:ivp}.
    Let $D=\sum_j\epsilon_j\mu_j$ be the associated divisor of the phase variable \(y\). Then the trace formula \(\Re[u]=-\frac{1}{2}Q_1(y)\) gives the real part of the solution \(u(x,t)\). Since \(\Lambda_{iu}\) is isospectral to \(\Lambda_u\) and $iu$ is also a solution of mKdV, running above argument again then gives the real part of the potential \(iu\), that is, \(\Im[u]=-\Re[iu] = \frac{1}{2}Q_1(\tilde{y})\), where \(\tilde{y}\) is the unique global solution of \eqref{eqn:ivp} with the initial condition replaced by \(y(0,0)=\tilde{f}=\mathcal{B}(i\varphi)\). The fact that the $u$ constructed in the above solves the Cauchy problem \eqref{eqn:mKdV}, \eqref{eqn:cauchyI} follows verbatim from \cite[Proposition 8.1]{FLLZ}.

    In order to prove the almost periodicity, we need to use the linearization property of the generalized Abel map by Theorem \ref{thm:linearAbel}. Moreover, according to \cite[Corollary 6.3, Theorem 6.4]{FLLZ}, there are frequencies \(\eta,\eta^{(1)}\in\R^\infty,\vartheta_0,\vartheta_1\in\R\) and vectors \(\zeta\in \T^\infty,\zeta'\in\T\) such that 
    \[(\zeta,\zeta')=\mathcal{A}\circ\mathcal{B}(q)\in\pi_1(\Omega)^*\times\T\simeq\T^\infty\times\T.\]
    Moreover, the translation and mKdV flow \(u(x,t)\mapsto u(x-\ell,t-\tau)\) is mapped to the following linear flow
    \[\mathcal{A}\circ\mathcal{B}(u(x-\ell,t-\tau))=\left(\zeta+\ell\eta+\tau\eta^{(1)},\zeta'+\ell\vartheta_0+\tau\vartheta_1\right).\]
    Since both maps \(\mathcal{A}\) and \(\mathcal{B}\) are homeomorphism, we have that \[u=\mathcal{M}\left(\zeta+\ell\eta+\tau\eta^{(1)},\zeta'+\ell\vartheta_0+\tau\vartheta_1\right)\]
    where \(\mathcal{M}=\mathcal{B}^{-1}\circ\mathcal{A}^{-1}\). This proves almost periodicity.
\end{proof}

 In order to prove Corollary \ref{cor:main2}, we need a direct spectral theory result in \cite{FLLZ}.
 \begin{lemma}[\cite{FLLZ}]\label{lem:ac}
     Under the conditions of Corollary~\ref{cor:main2}, the spectral type of $\Lambda_\varphi$ is purely absolutely continuous. For all $k\in\Z^d, r\in (0,h)$,
\[
\gamma_{k}\leq \epsilon_0^{\frac{1}{2}}e^{-2 \pi r |k|},
\] 
and for all $k' \neq k$,
\[\mathrm{dist}(G_k,G_{k'})\geq\frac{4^\tau\kappa^2}{c^2|k'-k|^{2\tau}}.
\]
Moreover, the spectrum is homogeneous in the sense of Carleson, and satisfies the Craig-type conditions \eqref{eqn:craigCond-1}  and \eqref{eqn:craigCond-3}.
 \end{lemma}
 Indeed, it is shown that there is a linear  (in \(k\)) upper bound for points in \(G_k\). This implies that \(\eta_{0k}<c|k|\). Interested reader can refer to the proof of \cite[Theorem 7.13]{FLLZ}.
 \begin{proof}[Proof of Corollary \ref{cor:main2}]
 Under the assumptions of the corollary, by Lemma \ref{lem:ac}, we have that the spectrum of \(\Lambda_\varphi\) is purely absolutely continuous and the spectrum \(\mathsf{E}\) is homogeneous satisfying the Craig type conditions. Therefore, the conditions of Theorem \ref{thm:main1} are satisfied. 
 \end{proof}

\begin{appendix}
\section{Non-Stationary AKNS Hierarchy} \label{sec:AKNS}

Let $U(z)=\begin{bmatrix}-iz& q \\ p & iz\end{bmatrix}$ and
$\widetilde{V}_{n+1}$ is defined as follows:
$$\widetilde{V}_{n+1}=i\begin{bmatrix}-\widetilde{G}_{n+1}(z)&\widetilde{F}_{n}(z)\\-\widetilde{H}_{n}(z)&\widetilde{G}_{n+1}(z)\end{bmatrix},$$
\begin{equation}\label{eq:AKNSCoefficient}\begin{aligned}
&\widetilde{F}_{n}(z)=\sum_{s=0}^{n}\widetilde{f}_{n-s}z^{s},~\widetilde{f}_{0}=-iq,\\
&\widetilde{G}_{n+1}=\sum_{s=0}^{n+1}\widetilde{g}_{n+1-s}z^s,~\widetilde{G}_0=1,\\
&\widetilde{H}_n=\sum_{s=0}^n\widetilde{h}_{n-s}z^s,~\widetilde{H}_0=ip,
\end{aligned}\end{equation}
with the time dependent coefficients $\{\widetilde{f}_\ell,\widetilde{g}_\ell,\widetilde{h}_\ell\}_{\ell\in\N_0}$ recursively defined as
\begin{equation}\label{eq:AKNSCoefficient1}\begin{aligned}
&\widetilde{f}_0=-iq,~\widetilde{g}_0=1,~\widetilde{h}_0=ip,\\
&\widetilde{f}_{\ell+1}=(i/2)f_{\ell,x}-iq\widetilde{g}_{\ell+1},~\ell\in\N_0,\\
&\widetilde{g}_{\ell+1,x}=p\widetilde{f}_\ell+q\widetilde{h}_\ell,~\ell\in\N_0,\\
&\widetilde{h}_{\ell+1}=-(i/2)\widetilde{h}_{\ell,x}+ip\widetilde{g}_{\ell+1},~\ell\in\N_0;
\end{aligned}\end{equation}

The first a few terms can be computed explicitly as follows
\begin{equation}\label{eq:InitialTerms}
\begin{aligned}
&\widetilde{f}_0=-iq,~\widetilde{f}_1=\frac{1}{2}\partial_x q+c_1(-iq),\\
&\widetilde{f}_2=\frac{i}{4}\partial^2_x q-\frac{i}{2}pq^2+c_1(\frac{1}{2}\partial_x q)+c_2(-iq),\\
&\widetilde{f}_3=-\frac{1}{8}\partial_x^3q+\frac{3}{4}pq\partial_xq+c_1(i\frac{\partial_x^2q}{4})+c_2(-\frac{ipq^2}{2})+c_3(-iq);\\
&\widetilde{g}_0=1,~\widetilde{g}_1=c_1,\\
&\widetilde{g}_2=\frac{1}{2}pq+c_2,\\
&\widetilde{g}_3=\frac{i}{4}(p\partial_xq-q\partial_xp)+c_2(\frac{pq}{2})+c_3;\\
&\widetilde{h}_0=ip,~\widetilde{h}_1=\frac{1}{2}\partial_x p+c_1(ip),\\
&\widetilde{h}_2=-\frac{i}{4}\partial_x^2p+\frac{i}{2}p^2q+c_1(\frac{1}{2}\partial_x p)+c_2(ip),\\
&\widetilde{h}_3=-\frac{1}{8}\partial_x^3p+\frac{3}{4}pq\partial_xp+c_2(\frac{\partial_x p+ip2q}{2})+c_3(ip).
\end{aligned}
\end{equation}
The constants $\{c_i:i\in\N_0\}$ with convention $c_0=1$ are integration constants, the case $c_i=0$ for all $i\in\N$ is called {\it homogeneous} AKNS hierarchy. $\widetilde{V}_{n+1}, U$ are differential expressions of order $n+1$ and $1$.
The interested reader may consult Section 3.2 and Section 3.4 of \cite{GH03} for more detailed explanations.
The time dependent AKNS hierarchy now reads as the following zero-curvature equation
\begin{equation}\label{eq:zeroCur}
\begin{aligned}
U_{t_n}-\widetilde{V}_{n+1,x}+[U,\widetilde{V}_{n+1}]&=0\\
-V_{n+1,x}+[U,V_{n+1}]&=0,
\end{aligned}
\end{equation}
where $U,V_{n+1}$ are the corresponding stationary version of \eqref{eq:AKNSCoefficient}, \eqref{eq:AKNSCoefficient1}. 
The mKdV equation \eqref{eqn:mKdV} corresponds to the hierarchy for which \(n=2\). The wide tilde in these notations is merely for the intention of distinguishing the homogeneous case \(c_i=0,i\geq 1\) and nonhomogeneous case.

As written, this construction of the AKNS hierarchy corresponds to the Lax operator
\[
M = \begin{bmatrix}
i \partial_x & - iq \\
ip & -i \partial_x
\end{bmatrix}
\]
To specialize it to the defocusing case so that it corresponds to the Lax operator
\[
\Lambda_u = \begin{bmatrix}
i \partial_x & u \\
\overline{u} & -i \partial_x
\end{bmatrix}
\]
one takes $q = iu$ and $p = -i \overline{u}$.
\end{appendix}




\end{document}